\documentclass[11pt]{amsart}

\usepackage[margin=1in]{geometry}
\usepackage{amsmath,amssymb,amsthm,mathtools}
\usepackage{xcolor}
\usepackage[colorlinks=true,linkcolor=blue!55!black,citecolor=blue!55!black,urlcolor=blue!55!black]{hyperref}
\usepackage{microtype}
\usepackage{float}
\usepackage{tikz}
\usetikzlibrary{arrows.meta,positioning}
\usepackage{subcaption}
\usepackage{enumitem}

\newcommand{\K}{\mathbb K}
\newcommand{\F}{\mathcal F}
\DeclareMathOperator{\indim}{indim}
\DeclareMathOperator{\link}{link}
\DeclareMathOperator{\del}{del}
\DeclareMathOperator{\depth}{depth}
\DeclareMathOperator{\pdim}{pdim}
\DeclareMathOperator{\reg}{reg}
\DeclareMathOperator{\Ind}{Ind}
\DeclareMathOperator{\Cl}{Cl}

\newtheorem{theorem}{Theorem}[section]
\newtheorem{lemma}[theorem]{Lemma}
\newtheorem{proposition}[theorem]{Proposition}
\newtheorem{corollary}[theorem]{Corollary}
\newtheorem{question}[theorem]{Question}
\theoremstyle{definition}
\newtheorem{definition}[theorem]{Definition}
\newtheorem{example}[theorem]{Example}
\newtheorem{remark}[theorem]{Remark}
\newcommand{\doi}[1]{\href{https://doi.org/#1}{doi:\nolinkurl{#1}}}

\title{Vertex Dismissibility and Scalability of Simplicial Complexes}
\author{Mohammed Rafiq Namiq}
\address{Department of Mathematics, College of Science, University of Sulaimani, Sulaymaniyah, Kurdistan Region, Iraq}
\email{mohammed.namiq@univsul.edu.iq}

\begin{document}
	
	\begin{abstract}
		We study three nonpure extensions of vertex decomposability, shellability, and Cohen--Macaulayness at the minimum facet dimension. For $b=\operatorname{indim}\Delta$, strong vertex dismissibility gives a deletion--link recursion on the complex itself in which deletion does not decrease $b$, while vertex dismissibility and scalability require the pure initial skeleton $\Delta^{[b]}$ to be vertex decomposable and shellable, respectively. We prove that a join is strongly vertex dismissible exactly when both factors are strongly vertex dismissible. We also show that strong vertex dismissibility implies that every skeleton $\Delta^k=\Delta^{[k]}$, $k\le b$, is vertex decomposable. When $\operatorname{indim}\Delta\leq1$, weak connectedness is equivalent to vertex dismissibility, scalability, and initially Cohen--Macaulayness over some, equivalently every, field. For quasi-forests, these conditions are also equivalent to strong vertex dismissibility. Our main graph result classifies independence complexes of pseudoforests by showing that strong vertex dismissibility, vertex dismissibility, scalability, and initially Cohen--Macaulayness over some or every field are equivalent and are characterized by a componentwise independent domination criterion. Finally, we connect the combinatorial properties of the pure initial skeleton with algebraic properties of its Alexander dual, focusing on vertex splittability, linear quotients, and linear resolutions in fixed degree.
	\end{abstract}
	
	\maketitle
	
	\section{Introduction}\label{sec:1}
	
	Vertex decomposability and shellability are important combinatorial properties of simplicial complexes. They provide recursive and ordering methods for studying their topology and commutative algebra. Provan and Billera introduced vertex decomposability for pure complexes \cite{ProvanBillera1980}. Bj\"orner and Wachs later developed vertex decomposability and shellability for nonpure complexes, and Stanley studied their relation to sequentially Cohen--Macaulayness \cite{BjornerWachs1996,BjornerWachs1997,Stanley1996}. Alexander duality connects these properties with monomial ideals. The Eagon--Reiner theorem characterizes Cohen--Macaulay complexes by linear resolutions of Alexander dual ideals \cite{EagonReiner1998}. Shellability corresponds to linear quotients of the Alexander dual ideal \cite{HerzogHibiZheng2004}, while vertex decomposability corresponds to vertex splittability \cite{MoradiKhoshAhang2016}.
	
	Let $\Delta$ be a finite simplicial complex and set
	\[
	b=\indim\Delta
	=\min\{\dim F:F\in\mathcal F(\Delta)\}.
	\]
	The pure initial skeleton $\Delta^{[b]}$ is generated by all $b$-dimensional faces of $\Delta$. It describes the structure of $\Delta$ at its minimum facet dimension, including faces contained in facets of larger dimension.
	
	Initially Cohen--Macaulay modules and simplicial complexes were introduced in \cite{Namiq2026ICM}. A simplicial complex $\Delta$ is initially Cohen--Macaulay over a field $\K$ when its Stanley--Reisner ring $\K[\Delta]$ is initially Cohen--Macaulay. Equivalently, its pure initial skeleton $\Delta^{[b]}$ is Cohen--Macaulay over $\K$. This motivates stronger combinatorial conditions at the same dimension, modeled on vertex decomposability and shellability.
	
	We call $\Delta$ vertex dismissible when $\Delta^{[b]}$ is vertex decomposable and \emph{scalable} when $\Delta^{[b]}$ is shellable. To obtain a recursion on the nonpure complex itself, we introduce strong vertex dismissibility. A vertex $x$ is dismissing if
	\[
	\indim\del_\Delta(x)\geq\indim\Delta.
	\]
	The complex $\Delta$ is strongly vertex dismissible if it is a simplex, or if it has a dismissing vertex $x$ such that both $\del_\Delta(x)$ and $\link_\Delta(x)$ are strongly vertex dismissible. Thus vertex dismissibility and scalability are determined by the pure initial skeleton, whereas strong vertex dismissibility is a recursion on the complex itself.
	
	A basic local result relates these two notions. Lemma~\ref{lem:2} proves that $x$ is dismissing in $\Delta$ exactly when it is a shedding vertex of $\Delta^{[b]}$. Corollary~\ref{cor:2} then places strong vertex dismissibility between ordinary vertex decomposability and vertex dismissibility, followed by scalability and initially Cohen--Macaulayness. The converses fail in general. Proposition~\ref{prop:2} shows that strong vertex dismissibility agrees with ordinary vertex decomposability for pure complexes.
	
	Theorem~\ref{thm:1} shows that a join is strongly vertex dismissible exactly when both factors have this property. Theorem~\ref{thm:2} gives a sufficient reduction using vertices that occur only in facets above the initial dimension. The main structural result, Theorem~\ref{thm:3}, shows that if $\Delta$ is strongly vertex dismissible and $b=\indim\Delta$, then $\Delta^{[k]}=\Delta^k$ is vertex decomposable for every $-1\leq k\leq b$.
	
	We also obtain equivalence results for several classes of complexes. We use the notion of weak connectedness from \cite{Namiq2026ICM}, where a complex is weakly connected when its pure initial skeleton is strongly connected. Theorem~\ref{thm:4} shows that when $\indim\Delta\leq1$, weak connectedness is equivalent to vertex dismissibility, scalability, and initially Cohen--Macaulayness. For quasi-forests, Theorem~\ref{thm:5} shows that these conditions are also equivalent to strong vertex dismissibility. Corollaries~\ref{cor:4} and~\ref{cor:5} give corresponding results for independence complexes of co-chordal and complete multipartite graphs.
	
	We next study independence complexes. If $G$ is a graph, the facets of $\Ind(G)$ are its maximal independent sets. Proposition~\ref{prop:6} gives $\indim\Ind(G)=i(G)-1$ and shows that a vertex $v$ is dismissing exactly when
	\[
	i(G-v)\geq i(G),
	\]
	where $i(G)$ is the independent domination number \cite{Haynes1998}. It also expresses the two recursive branches through $G-v$ and $G-N_G[v]$. A graph is independent domination vertex critical, or $i$-vertex-critical, if $i(G-v)<i(G)$ for every $v\in V(G)$ \cite{AnanchuenEtAl2018,EdwardsMacGillivray2012}. Hence no vertex of an $i$-vertex-critical graph is dismissing. Corollary~\ref{cor:6} shows that if such a graph has an edge, then its independence complex is not strongly vertex dismissible.
	
	Two related lines of work should be distinguished from the problem studied here. Castrill\'on, Cruz, and Reyes characterized ordinary vertex decomposability, shellability, and Cohen--Macaulayness of independence complexes of unicyclic graphs \cite{CastrillonCruzReyes2016}. Wu, Jiang, Nazari-Moghaddam, Sheikholeslami, Shao, and Volkmann studied independent domination stability for trees and unicyclic graphs \cite{WuEtAl2019}; in particular, they obtained the same condition modulo $3$ for cycles. These results concern different properties from the initial skeleton conditions and the strong recursion considered in this paper.
	
	Proposition~\ref{prop:7} gives a complete description for connected unicyclic graphs using rooted independent domination. A cycle of length $3q+1$ is the cyclic obstruction unless one of its vertices is dismissing in the corresponding attached rooted tree. Theorem~\ref{thm:6} extends this description componentwise to pseudoforests and gives equivalent conditions for strong vertex dismissibility, vertex dismissibility, scalability, and initially Cohen--Macaulayness. This pseudoforest criterion is the main graph result of the paper. In particular, Corollary~\ref{cor:8} yields the residue condition $m\not\equiv1\pmod3$ for cycles in the present setting.
	
	The algebraic part uses Alexander duality. Let $J=I_{\Delta^\vee}$ and $d=n-b-1$. For a monomial ideal $I$, write
	\[
	I_\ell=
	\bigl\langle
	u\in I:
	u\text{ is squarefree and }\deg u=\ell
	\bigr\rangle.
	\]
	Lemma~\ref{lem:8} gives $J_d=I_{(\Delta^{[b]})^\vee}$. Theorem~\ref{thm:7} identifies vertex dismissibility with vertex divisibility, equivalently with vertex splittability of $J_d$, and scalability with degree quotients, equivalently with linear quotients of $J_d$. Strong vertex dismissibility corresponds to the recursive Alexander dual notion of strong vertex divisibility.
	
	Degree resolutions were introduced in \cite{Namiq2026Skeletons}. An ideal $I$ has a degree resolution when $\reg I=\deg I$. Earlier work \cite{Namiq2026ICM} shows that initially Cohen--Macaulayness is equivalent to the Alexander dual ideal itself having a degree resolution. In the present notation, this is also equivalent to $J_d$ having a $d$-linear resolution. Corollary~\ref{cor:11} further shows that if $\Delta$ is strongly vertex dismissible, then every relevant ideal $J_{n-k-1}$ is vertex splittable, has linear quotients, and has a linear resolution.
	
	Finally, Corollary~\ref{cor:12} gives the equivalent depth, projective dimension, and Alexander dual conditions for initially Cohen--Macaulayness. In particular, Corollary~\ref{cor:13} shows that if $\Delta$ is strongly vertex dismissible, vertex dismissible, or scalable, then
	\[
	\depth\K[\Delta]=b+1,
	\qquad
	\pdim\K[\Delta]=n-b-1
	\]
	over every field.
	
	The examples show that the implications among strong vertex dismissibility, vertex dismissibility, scalability, and initially Cohen--Macaulayness are strict, with the field qualification in Example~\ref{ex:3}. They also show that strong vertex dismissibility is not preserved under arbitrary links and, in general, is not determined by the pure initial skeleton alone.
	
	The paper is organized as follows. Section~\ref{sec:2} gives the notation and background. Section~\ref{sec:3} develops strong vertex dismissibility, joins, and the reduction by free vertices above the initial dimension. Section~\ref{sec:4} studies the pure initial skeleton and the relations among vertex dismissibility, scalability, and initially Cohen--Macaulayness. Section~\ref{sec:5} contains the quasi-forest and pseudoforest results. Section~\ref{sec:6} develops the Alexander dual correspondences, Section~\ref{sec:7} gives the homological consequences, and Section~\ref{sec:8} discusses limitations and open problems.
	
	\section{Conventions and background}\label{sec:2}
	
	All simplicial complexes are finite and contain the empty face.  The complex $\{\emptyset\}$ is the $(-1)$-simplex and has dimension $-1$.  A simplex on a finite set $F$ is denoted by $2^F$; thus $2^\emptyset=\{\emptyset\}$.  The vertex set is
	\[
	V(\Delta)=\bigcup_{F\in\Delta}F.
	\]
	A larger finite set $X\supseteq V(\Delta)$ may be fixed as a ground set. Elements of $X\setminus V(\Delta)$ do not occur as vertices of $\Delta$.  Recursive operations are performed on vertices of $\Delta$, whereas Alexander duality is always taken over the stated ground set $X$.
	
	The set of facets of $\Delta$ is $\F(\Delta)$.  We use
	\[
	\dim\Delta=\max\{\dim F:F\in\F(\Delta)\},
	\qquad
	\indim\Delta=\min\{\dim F:F\in\F(\Delta)\}.
	\]
	For $x\in V(\Delta)$,
	\[
	\del_\Delta(x)=\{F\in\Delta:x\notin F\},
	\qquad
	\link_\Delta(x)=\{F\in\Delta:x\notin F,\ F\cup\{x\}\in\Delta\}.
	\]
	We use the same notation for the link of each face. For $-1\leq k\leq\dim\Delta$, the ordinary $k$-skeleton and the pure $k$-skeleton are
	\[
	\Delta^k=\{F\in\Delta:\dim F\leq k\},
	\qquad
	\Delta^{[k]}=\langle F\in\Delta:\dim F=k\rangle.
	\]
	Thus the facets of $\Delta^{[k]}$ are all $k$-faces of $\Delta$.  This is the standard pure skeleton convention.  In particular, $\Delta^{[b]}$ contains the $b$-faces lying inside higher dimensional facets; it is not generated only by the facets of $\Delta$ having dimension $b$.
	
	\begin{lemma}\label{lem:1}
		Let $b=\indim\Delta$.  Then $\Delta^k=\Delta^{[k]}$ for every $-1\leq k\leq b$.
	\end{lemma}
	
	\begin{proof}
		For $k=-1$, the statement follows directly from the definitions.  Let $0\leq k\leq b$ and let $A$ be a face of $\Delta$ with $\dim A\leq k$.  Extend $A$ to a facet $F$ of $\Delta$.  Since $\dim F\geq b\geq k$, there is a $k$-face $B$ with $A\subseteq B\subseteq F$.  Hence $A\in\Delta^{[k]}$.  The reverse inclusion follows from the definition.
	\end{proof}
	
	We use the Bj\"orner--Wachs convention for vertex decomposability.  A complex is vertex decomposable if it is a simplex, or if it has a vertex $x$ such that both $\del_\Delta(x)$ and $\link_\Delta(x)$ are vertex decomposable and no facet of $\link_\Delta(x)$ is a facet of $\del_\Delta(x)$.  Such an $x$ is a \emph{shedding vertex}.  For a pure complex this is the Provan--Billera recursion.  We use pure shellability in the standard facet order sense.  Whenever Cohen--Macaulayness is mentioned, it is taken over the displayed field $\K$.
	
	For a fixed ground set $X$, let $R=\K[x_v:v\in X]$ and write $x^A=\prod_{v\in A}x_v$.  The Alexander dual is denoted by $\Delta^{\vee}$.  The convention needed below is
	\[
	I_{\Delta^{\vee}}
	=\bigl(x^{X\setminus F}:F\in\F(\Delta)\bigr).
	\]
	For a simplex $2^X$ this is the unit ideal.  We use the usual conventions for the zero and unit ideals and for their linear resolutions.
	
	\section{Dismissing vertices and the recursive theory}\label{sec:3}
	
	This section introduces dismissing vertices and develops the recursive theory for strong vertex dismissibility through deletion and link.
	
	\begin{definition}\label{def:1}
		A vertex $x\in V(\Delta)$ is \emph{dismissing} if $\indim\del_\Delta(x)\geq\indim\Delta$.
	\end{definition}
	
	The following lemma gives the exact relation between the complex and its pure initial skeleton.
	
	\begin{lemma}\label{lem:2}
		Let $b=\indim\Delta$.  A vertex $x$ is dismissing in $\Delta$ if and only if $x$ is a shedding vertex of $\Delta^{[b]}$.
	\end{lemma}
	
	\begin{proof}
		Assume first that $x$ is dismissing.  Let $F$ be a $b$-face containing $x$. The face $F\setminus\{x\}$ belongs to $\del_\Delta(x)$.  It extends to a facet $Q$ of the deletion, and $\dim Q\geq b$.  Choose a $b$-face $F'$ of $Q$ containing $F\setminus\{x\}$.  Then $F'$ does not contain $x$ and is a facet of $\Delta^{[b]}$.  This is the shedding condition in the pure $b$-skeleton.
		
		Conversely, suppose that $x$ is shedding in $\Delta^{[b]}$.  If $\del_\Delta(x)$ had a facet $Q$ with $\dim Q<b$, extend $Q$ to a facet $Q'$ of $\Delta$.  Maximality of $Q$ among the faces not containing $x$ implies that $Q'=Q\cup\{x\}$.  Since every facet of $\Delta$ has dimension at least $b$, we obtain $\dim Q=b-1$ and $\dim Q'=b$.  The shedding condition gives a $b$-face not containing $x$ and containing $Q$, which contradicts the assumption that $Q$ is a facet of the deletion.
	\end{proof}
	
	\begin{corollary}\label{cor:1}
		Every shedding vertex of $\Delta$ is dismissing.  If $\Delta$ is pure, then a vertex is dismissing if and only if it is shedding.
	\end{corollary}
	
	\begin{proof}
		Suppose that $x$ is shedding in $\Delta$.  If the deletion had a facet $Q$ of dimension below $b=\indim\Delta$, extend $Q$ to a facet $Q'$ of $\Delta$.  Maximality of $Q$ in the deletion implies that $Q'=Q\cup\{x\}$.  Thus $Q$ is a face of the link.  It is a facet of the link as well, because a larger link face would be a larger face of the deletion.  This contradicts the shedding condition.  Hence $x$ is dismissing.  If $\Delta$ is pure of dimension $b$, then $\Delta^{[b]}=\Delta$, so the second statement follows from Lemma~\ref{lem:2}.
	\end{proof}
	
	\begin{example}\label{ex:1}
		Let $\Delta=\langle xa,ab,xcd\rangle$. Then $\indim\Delta=1$ and $\del_\Delta(x)=\langle ab,cd\rangle$, so $x$ is dismissing.  It is not a shedding vertex of $\Delta$, because $cd$ is a facet of both $\link_\Delta(x)$ and $\del_\Delta(x)$.  By Lemma~\ref{lem:2}, $x$ is a shedding vertex of the pure initial skeleton.
	\end{example}
	
	\begin{definition}\label{def:2}
		A simplicial complex $\Delta$ is \emph{strongly vertex dismissible} if either $\Delta$ is a simplex, or there is a dismissing vertex $x$ such that $\del_\Delta(x)$ and $\link_\Delta(x)$ are both strongly vertex dismissible.
	\end{definition}
	
	This recursion differs from weak $k$-decomposability and weak vertex decomposability. It uses both deletion and link in the recursion, but the vertex is chosen using the minimum facet dimension rather than the shedding condition on the entire complex \cite{NagelRomer2008,ProvanBillera1980}.
	
	\begin{proposition}\label{prop:1}
		Every vertex decomposable complex is strongly vertex dismissible.
	\end{proposition}
	
	\begin{proof}
		We induct on the number of vertices.  If $\Delta$ is a simplex, the assertion follows. Otherwise choose a shedding vertex $x$.  It is dismissing by Corollary~\ref{cor:1}, and the deletion and link are strongly vertex dismissible by induction.
	\end{proof}
	
	\begin{proposition}\label{prop:2}
		A pure complex is strongly vertex dismissible if and only if it is vertex decomposable.
	\end{proposition}
	
	\begin{proof}
		One direction is Proposition~\ref{prop:1}.  For the converse, we induct on the number of vertices.  If $x$ begins a strong dismissal of a pure $p$-dimensional complex, then $x$ is shedding by Corollary~\ref{cor:1}.  Its deletion is pure of dimension $p$, its link is pure of dimension $p-1$, and both are strongly vertex dismissible.  By induction, the deletion and link are vertex decomposable, so $x$ begins an ordinary vertex decomposition.
	\end{proof}
	
	\begin{example}\label{ex:2}
		Let $\Delta=\langle a,bc,de\rangle$. The vertices $b,a,d,c$ may be deleted in that order; at every step the link is a simplex, and the deleting vertex is dismissing.  Thus $\Delta$ is strongly vertex dismissible.  On the other hand, only $a$ can begin an ordinary vertex decomposition.  Its deletion is the pure complex $\langle bc,de\rangle$, a disjoint union of two edges, which has no shedding vertex.  Hence $\Delta$ is not vertex decomposable.
	\end{example}
	
	\subsection{Joins}
	
	The join is taken on disjoint vertex sets.  Its facets are the unions of facets of the factors, and therefore
	\[
	\indim(\Delta*\Gamma)
	=\indim\Delta+\indim\Gamma+1.
	\]
	
	\begin{theorem}\label{thm:1}
		Let $\Delta$ and $\Gamma$ be finite complexes on disjoint vertex sets.  Then the join $\Delta*\Gamma$ is strongly vertex dismissible if and only if $\Delta$ and $\Gamma$ are strongly vertex dismissible.
	\end{theorem}
	
	\begin{proof}
		For $x\in V(\Delta)$,
		\[
		\del_{\Delta*\Gamma}(x)=\del_\Delta(x)*\Gamma,
		\qquad
		\link_{\Delta*\Gamma}(x)=\link_\Delta(x)*\Gamma.
		\]
		The formula for the initial dimension shows that $x$ is dismissing in the join if and only if it is dismissing in $\Delta$.  The same statement holds for a vertex of $\Gamma$.
		
		Induct on $|V(\Delta)|+|V(\Gamma)|$.  Suppose first that both factors are strongly vertex dismissible.  If both are simplices, then their join is a simplex.  Otherwise choose a first dismissing vertex in a factor that is not a simplex, say $x\in V(\Delta)$.  Both joins displayed above are strongly vertex dismissible by induction, and $x$ is dismissing in the join.  Hence $\Delta*\Gamma$ is strongly vertex dismissible.
		
		Conversely, suppose that $\Delta*\Gamma$ is strongly vertex dismissible.  If it is a simplex, then each factor is a simplex.  Otherwise let $x$ begin a strong dismissal of the join.  If $x\in V(\Delta)$, induction applied to the deletion and link gives
		\[
		\del_\Delta(x),\quad \link_\Delta(x),\quad \Gamma
		\]
		strongly vertex dismissible.  Since $x$ is dismissing in $\Delta$, it begins a strong dismissal of $\Delta$.  Thus both factors have the property.  The case $x\in V(\Gamma)$ is symmetric.
	\end{proof}
	
	\subsection{Reductions by free vertices above the initial dimension}
	
	\begin{definition}\label{def:3}
		Let $b=\indim\Delta$.  A vertex $x$ is \emph{free above the initial dimension} if it is contained in exactly one facet $F$ and $\dim F>b$.  A reduction by such vertices is a sequence in which, after each deletion, the next vertex is free above the initial dimension in the complex that remains.
	\end{definition}
	
	\begin{theorem}\label{thm:2}
		Suppose that repeatedly deleting vertices that are free above the initial dimension ends in a pure vertex decomposable complex.  Then $\Delta$ is strongly vertex dismissible.
	\end{theorem}
	
	\begin{proof}
		Let $x$ be free above the initial dimension, and let $F$ be its unique containing facet.  No facet of minimum dimension contains $x$, so deletion preserves the initial dimension.  The only possible new facet is $F\setminus\{x\}$, whose dimension is at least the old initial dimension.  Thus $x$ is dismissing. Also,
		\[
		\link_\Delta(x)=2^{F\setminus\{x\}}
		\]
		is a simplex.  Start from the final pure vertex decomposable complex, which is strongly vertex dismissible by Proposition~\ref{prop:1}, and reverse the reduction.  At each step the deletion branch has already been verified and the link branch is a simplex.
	\end{proof}
	
	\begin{remark}\label{rem:1}
		The theorem does not imply that repeated deletion of such free vertices must end at the original pure initial skeleton. Deleting a vertex from a facet above the initial dimension may create a new facet, and the same free vertex condition need not continue to hold until all facets of larger dimension have been removed.
	\end{remark}
	
	\section{The pure initial skeleton and its skeleta}\label{sec:4}
	
	We first recall a classical result about pure skeleta \cite{ProvanBillera1980}. We include the simplex base case explicitly because a lower pure skeleton of a simplex is usually not itself a simplex.
	
	\begin{lemma}\label{lem:3}
		Let $\Gamma$ be a pure vertex decomposable complex of dimension $p$.  Then $\Gamma^{[k]}$ is vertex decomposable for every $-1\leq k\leq p$.
	\end{lemma}
	
	\begin{proof}
		The cases $k=-1$ and $k=p$ follow directly.  We now induct on the number of vertices.  If $\Gamma=2^F$ is a simplex and $k<p$, then $\Gamma^{[k]}$ is the uniform matroid complex consisting of all subsets of $F$ of size at most $k+1$.  Every vertex is shedding; its deletion and link are smaller uniform matroid complexes.  This proves the simplex case by induction on $|F|$.
		
		Now suppose that $\Gamma$ is not a simplex and let $x$ be a shedding vertex in a pure vertex decomposition.  For $k<p$,
		\[
		\del_{\Gamma^{[k]}}(x)=(\del_\Gamma(x))^{[k]},
		\qquad
		\link_{\Gamma^{[k]}}(x)=(\link_\Gamma(x))^{[k-1]}.
		\]
		To see that $x$ is shedding in $\Gamma^{[k]}$, take a $k$-face $A$ containing $x$ and extend it to a facet $Q$ of $\Gamma$.  Since $x$ is shedding in the pure complex $\Gamma$, the face $Q\setminus\{x\}$ lies in a facet not containing $x$.  A $k$-face of that facet can be chosen to contain $A\setminus\{x\}$.  The displayed deletion and link are vertex decomposable by induction, so $\Gamma^{[k]}$ is vertex decomposable.
	\end{proof}
	
	\begin{theorem}\label{thm:3}
		Let $\Delta$ be strongly vertex dismissible and let $b=\indim\Delta$.  Then $\Delta^{[b]}$ is vertex decomposable.  Therefore, $\Delta^{[k]}=\Delta^k$ is vertex decomposable for every $-1\leq k\leq b$.
	\end{theorem}
	
	\begin{proof}
		We first prove the result for the pure initial skeleton by induction on the number of vertices.  If $\Delta$ is a simplex, then $\Delta^{[b]}=\Delta$.  Otherwise let $x$ begin a strong dismissal, and put
		\[
		\Delta'=\del_\Delta(x),\qquad \Lambda=\link_\Delta(x),
		\]
		with $b_{\Delta'}=\indim\Delta'$ and $b_\Lambda=\indim\Lambda$.  The dismissing condition gives $b_{\Delta'}\geq b$, and every facet of $\Lambda$ comes from a facet of $\Delta$ containing $x$, so $b_\Lambda\geq b-1$.  By Lemma~\ref{lem:2}, $x$ is a shedding vertex of $\Delta^{[b]}$.
		
		The induction hypothesis gives vertex decomposability of $(\Delta')^{[b_{\Delta'}]}$ and $\Lambda^{[b_\Lambda]}$.  Since $b\leq b_{\Delta'}$ and $b-1\leq b_\Lambda$, Lemma~\ref{lem:3} gives vertex decomposability of $(\Delta')^{[b]}$ and $\Lambda^{[b-1]}$.  Also,
		\[
		\del_{\Delta^{[b]}}(x)=(\Delta')^{[b]},
		\qquad
		\link_{\Delta^{[b]}}(x)=\Lambda^{[b-1]}.
		\]
		Thus $x$ begins a vertex decomposition of $\Delta^{[b]}$. For $-1\leq k\leq b$, every $k$-face of $\Delta$ lies in a $b$-face, so
		\[
		\Delta^{[k]}=(\Delta^{[b]})^{[k]}=\Delta^k.
		\]
		The lower rank conclusion now follows from Lemma~\ref{lem:3} applied to the pure vertex decomposable complex $\Delta^{[b]}$.
	\end{proof}
	
	\begin{remark}\label{rem:2}
		The main implication in Theorem~\ref{thm:3} is that strong vertex dismissibility implies that the pure initial skeleton is vertex decomposable. The conclusion for lower ranks then follows from the classical fact that pure skeleta of smaller dimension of a vertex decomposable complex are vertex decomposable; see, for example, \cite[Lemma~3.10]{Woodroofe2011}.  In particular, the same lower rank conclusion holds under the weaker hypothesis that $\Delta^{[b]}$ is vertex decomposable.
	\end{remark}
	
	\subsection{Vertex dismissibility, scalability, and initially Cohen--Macaulayness}
	
	\begin{definition}\label{def:4}
		Let $b=\indim\Delta$.
		\begin{enumerate}[label=\textup{(\roman*)}]
			\item $\Delta$ is \emph{vertex dismissible} if $\Delta^{[b]}$ is vertex decomposable.
			\item $\Delta$ is \emph{scalable} if $\Delta^{[b]}$ is shellable.
			\item Following \cite{Namiq2026ICM}, $\Delta$ is \emph{initially Cohen--Macaulay over $\K$} if its Stanley--Reisner ring $\K[\Delta]$ is initially Cohen--Macaulay.  By \cite[Proposition~4.6]{Namiq2026ICM} and Lemma~\ref{lem:1}, this is equivalent to $\Delta^{[b]}$ being Cohen--Macaulay over $\K$.
		\end{enumerate}
	\end{definition}
	
	The term \emph{scalable} means that the pure initial skeleton is shellable. Thus scalability agrees with ordinary shellability for a pure complex. The notion of an \emph{initially Cohen--Macaulay} module was introduced in \cite{Namiq2026ICM} by the condition $\depth N=\indim N$, and the same work defines the corresponding property for simplicial complexes through their Stanley--Reisner rings. Proposition~4.6 of that paper identifies the property with Cohen--Macaulayness of the initial dimension skeleton. By Lemma~\ref{lem:1}, this skeleton is $\Delta^{[b]}$ when $b=\indim\Delta$. Corollary~\ref{cor:12} gives the equivalent depth and Alexander dual formulations.
	
	\begin{corollary}\label{cor:2}
		For every field $\K$,
		\[
		\begin{split}
			\text{vertex decomposable}
			&\Longrightarrow \text{strongly vertex dismissible}
			\Longrightarrow \text{vertex dismissible}\Longrightarrow \text{scalable}\\
			&\Longrightarrow \text{initially Cohen--Macaulay over }\K.
		\end{split}
		\]
		For pure complexes, the first three properties are equivalent, scalability is ordinary shellability, and initially Cohen--Macaulayness is ordinary Cohen--Macaulayness.
	\end{corollary}
	
	\begin{proof}
		The first implication follows from Proposition~\ref{prop:1}.  The second implication follows from Theorem~\ref{thm:3} at $k=b$.  The remaining implications follow from the classical pure hierarchy.  The statement for pure complexes follows from Proposition~\ref{prop:2}.
	\end{proof}
	
	\begin{example}\label{ex:3}
		Let $W=\{a,b,c,d,e,f,g\}$ and
		\begin{align*}
			\F(\Delta_1)&=\{abcd,cdef,acg,aef,abe,afg,bdg,beg,ceg,dfg\},\\
			\F(\Delta_2)&=\{abcd,cdef,acg,aef,abe,abf,bfg,bdg,beg,ceg,dfg\}.
		\end{align*}
		Thus $\Delta_2$ is obtained from $\Delta_1$ by replacing the facet $afg$ with the two facets $abf$ and $bfg$. Put $\Gamma_j=\Delta_j^{[2]}$ for $j=1,2$. The facet lists define the complexes. Figure~\ref{fig:1} provides a visual representation, and repeated labels denote the same vertex.
		
		\begin{figure}[h]
			\centering
			\tikzset{
				wire/.style={thick,black!80},
				bead/.style={circle,fill=black!70,inner sep=0pt,minimum size=4.5pt},
				label node/.style={font=\scriptsize,inner sep=1.5pt,text=black}
			}
			\begin{subfigure}[b]{0.48\textwidth}
				\centering
				\begin{tikzpicture}[scale=0.55,line join=round,line cap=round]
					\coordinate (c) at (-1.732,0); \coordinate (d) at (1.732,0);
					\coordinate (g) at (0,3); \coordinate (b) at (0,1);
					\coordinate (f) at (0,-3); \coordinate (e) at (0,-1);
					\coordinate (a) at (-4,1); \coordinate (h) at (5.3,0);
					\coordinate (i) at (1,5); \coordinate (j) at (-1.5,3);
					\fill[black!18] (a)--(c)--(f)--cycle;
					\fill[black!22] (a)--(c)--(g)--cycle;
					\fill[black!13] (a)--(g)--(j)--cycle;
					\fill[black!10] (i)--(j)--(g)--cycle;
					\fill[black!15] (i)--(b)--(g)--cycle;
					\fill[black!18] (h)--(i)--(b)--cycle;
					\fill[black!13] (h)--(b)--(d)--cycle;
					\fill[black!13] (h)--(e)--(d)--cycle;
					\fill[black!35] (b)--(c)--(g)--cycle;
					\fill[black!30] (b)--(d)--(g)--cycle;
					\fill[black!25] (b)--(c)--(d)--cycle;
					\fill[black!35] (e)--(c)--(f)--cycle;
					\fill[black!30] (e)--(d)--(f)--cycle;
					\fill[black!25] (e)--(c)--(d)--cycle;
					\draw[wire] (a)--(j)--(i)--(h);
					\draw[wire] (a)--(g)--(i);
					\draw[wire] (a)--(c)--(g)--(d)--(h);
					\draw[wire] (a)--(f)--(d);
					\draw[wire] (i)--(b)--(h);
					\draw[wire] (h)--(e);
					\draw[wire] (c)--(d);
					\draw[wire] (c)--(f);
					\draw[wire] (e)--(c);
					\foreach \v in {a,b,c,d,e,f,g,h,i,j}{\node[bead] at (\v) {};}
					\node[label node,above=2pt] at (g) {$e$};
					\node[label node,below=2pt] at (f) {$a$};
					\node[label node,left=3pt] at (c) {$c$};
					\node[label node,right=3pt] at (d) {$d$};
					\node[label node,right=2pt] at (b) {$f$};
					\node[label node,right=2pt] at (e) {$b$};
					\node[label node,left=3pt] at (a) {$g$};
					\node[label node,right=3pt] at (h) {$g$};
					\node[label node,above=3pt] at (i) {$a$};
					\node[label node,left=4pt] at (j) {$b$};
				\end{tikzpicture}
				\caption{$\Delta_1$ is initially Cohen--Macaulay when
					$\operatorname{char}\K\ne2$, but is not scalable.}
			\end{subfigure}
			\hfill
			\begin{subfigure}[b]{0.48\textwidth}
				\centering
				\begin{tikzpicture}[scale=0.55,line join=round,line cap=round]
					\coordinate (c) at (-1.732,0); \coordinate (d) at (1.732,0);
					\coordinate (g) at (0,3); \coordinate (b) at (0,1);
					\coordinate (f) at (0,-3); \coordinate (e) at (0,-1);
					\coordinate (a) at (-4,1); \coordinate (h) at (5.3,0);
					\coordinate (i) at (1,5); \coordinate (j) at (-1.5,3);
					\coordinate (k) at (2.75,3);
					\fill[black!18] (a)--(c)--(f)--cycle;
					\fill[black!22] (a)--(c)--(g)--cycle;
					\fill[black!13] (a)--(g)--(j)--cycle;
					\fill[black!10] (i)--(j)--(g)--cycle;
					\fill[black!15] (i)--(b)--(g)--cycle;
					\fill[black!18] (i)--(b)--(k)--cycle;
					\fill[black!18] (h)--(b)--(k)--cycle;
					\fill[black!13] (h)--(b)--(d)--cycle;
					\fill[black!13] (h)--(e)--(d)--cycle;
					\fill[black!35] (b)--(c)--(g)--cycle;
					\fill[black!30] (b)--(d)--(g)--cycle;
					\fill[black!25] (b)--(c)--(d)--cycle;
					\fill[black!35] (e)--(c)--(f)--cycle;
					\fill[black!30] (e)--(d)--(f)--cycle;
					\fill[black!25] (e)--(c)--(d)--cycle;
					\draw[wire] (a)--(j)--(i)--(k)--(h);
					\draw[wire] (a)--(g)--(i);
					\draw[wire] (a)--(c)--(g)--(d)--(h);
					\draw[wire] (a)--(f)--(d);
					\draw[wire] (i)--(b)--(k);
					\draw[wire] (b)--(h);
					\draw[wire] (h)--(e);
					\draw[wire] (c)--(d);
					\draw[wire] (c)--(f);
					\draw[wire] (e)--(c);
					\foreach \v in {a,b,c,d,e,f,g,h,i,j,k}{\node[bead] at (\v) {};}
					\node[label node,above=2pt] at (g) {$e$};
					\node[label node,below=2pt] at (f) {$a$};
					\node[label node,left=3pt] at (c) {$c$};
					\node[label node,right=3pt] at (d) {$d$};
					\node[label node,right=2pt] at (b) {$f$};
					\node[label node,right=2pt] at (e) {$b$};
					\node[label node,left=3pt] at (a) {$g$};
					\node[label node,right=3pt] at (h) {$g$};
					\node[label node,above=3pt] at (i) {$a$};
					\node[label node,left=4pt] at (j) {$b$};
					\node[label node,right=4pt] at (k) {$b$};
				\end{tikzpicture}
				\caption{$\Delta_2$ is scalable, but is not vertex dismissible.\\~}
			\end{subfigure}
			\caption{Strictness of vertex dismissibility $\Longrightarrow$ scalability $\Longrightarrow$ initially Cohen--Macaulayness.}
			\label{fig:1}
		\end{figure}
		
		An integral boundary computation gives
		\[
		\widetilde H_1(\Gamma_1;\mathbb Z)
		\cong \mathbb Z/2\mathbb Z,
		\]
		while $\Gamma_1$ and all of its vertex links are connected. By Reisner's criterion \cite{Reisner1976}, together with $\Gamma_1=\Delta_1^{[2]}$, it follows that $\Delta_1$ is initially Cohen--Macaulay exactly when $\operatorname{char}\K\ne2$. In particular, $\Delta_1$ is not initially Cohen--Macaulay in characteristic two. Since scalability implies the initially Cohen--Macaulay property over every field, $\Delta_1$ is not scalable.
		
		For $\Delta_2$, the following facet order of $\Gamma_2$ certifies scalability:
		\[
		\begin{aligned}
			&beg,ceg,abe,abf,aef,bfg,cef,bdg,dfg,\\
			&cdf,bcd,abc,def,acd,abd,acg,cde.
		\end{aligned}
		\]
		We next show that $\Delta_2$ is not vertex dismissible, equivalently, that the pure complex $\Gamma_2$ is not vertex decomposable. The vertex $c$ is not a shedding vertex of $\Gamma_2$: the edge $ag$ is a facet of both $\link_{\Gamma_2}(c)$ and $\del_{\Gamma_2}(c)$, so the shedding condition fails. For $v\in\{a,b,d,e,f,g\}$, the deletion is pure and satisfies
		\[
		\begin{array}{c|cccccc}
			v&a&b&d&e&f&g\\ \hline
			\dim_{\mathbb Q}\widetilde H_1(\del_{\Gamma_2}(v);\mathbb Q)
			&1&2&1&1&1&1.
		\end{array}
		\]
		Hence every one of these deletions is not Cohen--Macaulay over $\mathbb Q$, and therefore is not vertex decomposable. Consequently no vertex can start a vertex decomposition of $\Gamma_2$. Thus $\Gamma_2$ is not vertex decomposable and $\Delta_2$ is not vertex dismissible. We have therefore shown that both implications in Corollary~\ref{cor:2},
		\[
		\text{vertex dismissibility}
		\Longrightarrow
		\text{scalability}
		\Longrightarrow
		\text{initially Cohen--Macaulay},
		\]
		are strict when $\operatorname{char}\K\ne2$, while the first implication is strict over every field.
		
		If $L_j=I_{\Delta_j^{\vee}}$, then Theorem~\ref{thm:7} shows that the corresponding implications among vertex divisibility, degree quotients, and degree resolutions are also strict, with the same field qualification, for the equigenerated ideals $(L_j)_4$; in one degree these reduce to vertex splittability, linear quotients, and linear resolutions, respectively.
	\end{example}
	
	\begin{proposition}\label{prop:3}
		Let $b=\indim\Delta$ and fix a field $\K$.  For each of the properties
		\[
		\mathcal P\in\{\text{vertex decomposable},\ \text{shellable},\
		\text{Cohen--Macaulay over }\K\},
		\]
		the following are equivalent:
		\begin{enumerate}[label=\textup{(\roman*)}]
			\item $\Delta^{[b]}$ has property $\mathcal P$;
			\item $\Delta^k=\Delta^{[k]}$ has property $\mathcal P$ for every $-1\leq k\leq b$.
		\end{enumerate}
	\end{proposition}
	
	\begin{proof}
		Taking $k=b$ proves one direction.  Lemma~\ref{lem:3} proves the other direction for vertex decomposability.  Lower pure skeleta of a pure shellable complex are shellable, and lower skeleta of a pure Cohen--Macaulay complex are Cohen--Macaulay; these are standard rank selection facts in the pure theory \cite{BjornerWachs1996,Namiq2026Skeletons,ProvanBillera1980}. Use Lemma~\ref{lem:1} to identify the ordinary and pure skeleta.
	\end{proof}
	
	\subsection{Links and joins of the pure initial skeleton}
	
	\begin{proposition}\label{prop:4}
		Let $b=\indim\Delta$, and let $\sigma$ be contained in a $b$-dimensional facet of $\Delta$.  If $\Delta^{[b]}$ is vertex decomposable, shellable, or Cohen--Macaulay over $\K$, then the pure initial skeleton of $\link_\Delta(\sigma)$ has the corresponding property.
	\end{proposition}
	
	\begin{proof}
		Because $\sigma$ lies in a $b$-dimensional facet,
		\[
		\indim\link_\Delta(\sigma)=b-|\sigma|
		\]
		and
		\[
		\bigl(\link_\Delta(\sigma)\bigr)^{[b-|\sigma|]}
		=\link_{\Delta^{[b]}}(\sigma).
		\]
		Each of these three properties is preserved under links of a pure complex.
	\end{proof}
	
	\begin{proposition}\label{prop:5}
		Let $\Delta$ and $\Gamma$ be finite complexes on disjoint vertex sets.  If the pure initial skeleton of $\Delta*\Gamma$ is vertex decomposable, shellable, or Cohen--Macaulay over $\K$, then the pure initial skeleta of both factors have the corresponding property.
	\end{proposition}
	
	\begin{proof}
		Set $b_\Delta=\indim\Delta$, $b_\Gamma=\indim\Gamma$, and
		\[
		\Omega=(\Delta*\Gamma)^{[b_\Delta+b_\Gamma+1]}.
		\]
		Choose a $b_\Gamma$-dimensional facet $F_\Gamma$ of $\Gamma$.  Since $F_\Gamma$ is maximal in $\Gamma$,
		\[
		\link_\Omega(F_\Gamma)=\Delta^{[b_\Delta]}.
		\]
		The three properties are inherited by links.  Exchanging the two factors gives the result for $\Gamma$.
	\end{proof}
	
	\begin{remark}\label{rem:3}
		Proposition~\ref{prop:5} does not have a converse based only on the pure initial skeleta of the two factors. The pure initial skeleton of a join can contain faces whose ranks are distributed between the two factors in several ways, so it need not equal the join of the two pure initial skeleta.  By contrast, Theorem~\ref{thm:1} gives an equivalence for strong vertex dismissibility.
	\end{remark}
	
	\subsection{The pure initial skeleton does not determine the recursion}
	
	\begin{lemma}\label{lem:4}
		Every connected pure one-dimensional simplicial complex is vertex decomposable.
	\end{lemma}
	
	\begin{proof}
		Induct on the number of vertices.  A single edge is a simplex.  Otherwise choose a vertex $x$ that is not a cut vertex.  The deletion remains connected and has at least two vertices.  Hence no neighbor of $x$ becomes isolated in the deletion, which is exactly the shedding condition in a pure graph.  The link is zero-dimensional and the deletion is vertex decomposable by induction.
	\end{proof}
	
	\begin{example}\label{ex:4}
		Let $\Gamma=\langle dg,abef,aceg,bcfg\rangle$. Its initial dimension is $1$, and $\Gamma^{[1]}$ is a connected graph. Thus $\Gamma$ is vertex dismissible by Lemma~\ref{lem:4}.
		
		No vertex can begin a strong dismissal.  The vertex $g$ is not dismissing, because $\del_\Gamma(g)$ has the singleton facet $\{d\}$.  For $v\in\{a,b,c,e,f\}$, the link has two two-dimensional facets that meet in exactly one vertex:
		\[
		\begin{array}{c|c}
			v & \F(\link_\Gamma(v))\\ \hline
			a & bef,\ ceg\\
			b & aef,\ cfg\\
			c & aeg,\ bfg\\
			e & abf,\ acg\\
			f & abe,\ bcg
		\end{array}
		\]
		In each case the two facets share no edge, so the facet-ridge graph is disconnected.  Hence the pure two-dimensional link is not shellable, and therefore it is neither vertex decomposable nor strongly vertex dismissible by Proposition~\ref{prop:2}.  The remaining vertex $d$ is dismissing, but
		\[
		\del_\Gamma(d)=\langle abef,aceg,bcfg\rangle
		\]
		is pure of dimension $3$ and its three facets meet pairwise only in edges. Its facet-ridge graph is disconnected; hence it is not shellable and therefore not vertex decomposable or strongly vertex dismissible.  Thus $\Gamma$ is not strongly vertex dismissible.
	\end{example}
	
	\section{Equivalence results for selected classes}
	\label{sec:5}
	
	The converses in Corollary~\ref{cor:2} fail in general.  This section gives three classes in which several of the properties are equivalent.  We first treat complexes of initial dimension at most one, then quasi-forests, and finally independence complexes of pseudoforests.
	
	Let $b=\indim\Delta$.  Following \cite{Namiq2026ICM}, we say that $\Delta$ is \emph{weakly connected} if the facet-ridge graph of the pure initial skeleton $\Delta^{[b]}$ is connected.  Equivalently, the graph with vertex set $\F(\Delta)$ in which two facets $F_1$ and $F_2$ are adjacent when
	\[
	|F_1\cap F_2|\geq b
	\]
	is connected.  For $b=0$, this convention makes every nonempty pure zero-dimensional complex weakly connected.
	
	\subsection{Initial dimensions zero and one}
	\label{subsec:1}
	
	\begin{theorem}\label{thm:4}
		Let $\Delta$ be a simplicial complex with $\indim\Delta\leq1$, and put $b=\indim\Delta$.  The following are equivalent:
		\begin{enumerate}[label=\textup{(\roman*)}]
			\item $\Delta$ is weakly connected;
			\item $\Delta$ is vertex dismissible;
			\item $\Delta$ is scalable;
			\item $\Delta$ is initially Cohen--Macaulay over every field;
			\item $\Delta$ is initially Cohen--Macaulay over some field.
		\end{enumerate}
	\end{theorem}
	
	\begin{proof}
		Put $b=\indim\Delta$.  If $b=-1$, then $\Delta=\{\emptyset\}$, so all five conditions hold under the standard conventions.  If $b=0$, then $\Delta^{[0]}$ is a nonempty zero-dimensional complex.  It is vertex decomposable, shellable, and Cohen--Macaulay over every field, and it is weakly connected under the convention above.
		
		Suppose that $b=1$.  Then $\Delta^{[1]}$ is a pure graph, and weak connectedness is ordinary connectivity of this graph.  By Lemma~\ref{lem:4}, every connected pure graph is vertex decomposable, so (i) implies (ii).  The implications
		\[
		\textup{(ii)}\Longrightarrow\textup{(iii)}
		\Longrightarrow\textup{(iv)}\Longrightarrow\textup{(v)}
		\]
		follow from the classical pure hierarchy.  Finally, Reisner's criterion in dimension one implies that a pure Cohen--Macaulay graph is connected.  Hence (v) implies (i).
	\end{proof}
	
	\begin{corollary}\label{cor:3}
		For a pure one-dimensional simplicial complex $\Gamma$, the following are equivalent:
		\begin{enumerate}[label=\textup{(\roman*)}]
			\item $\Gamma$ is connected;
			\item $\Gamma$ is vertex decomposable;
			\item $\Gamma$ is shellable;
			\item $\Gamma$ is Cohen--Macaulay over some field;
			\item $\Gamma$ is Cohen--Macaulay over every field.
		\end{enumerate}
	\end{corollary}
	
	\begin{proof}
		Apply Theorem~\ref{thm:4} to the pure complex $\Gamma$.  Since $\Gamma^{[1]}=\Gamma$, vertex dismissibility is ordinary vertex decomposability, scalability is ordinary shellability, and initially Cohen--Macaulayness is ordinary Cohen--Macaulayness.
	\end{proof}
	
	\begin{remark}\label{rem:4}
		Strong vertex dismissibility is not included in Theorem~\ref{thm:4}. Example~\ref{ex:4} has initial dimension one and is weakly connected, vertex dismissible, scalable, and initially Cohen--Macaulay, but it is not strongly vertex dismissible.
	\end{remark}
	
	\subsection{Quasi-forests}
	\label{subsec:2}
	
	A facet $F$ of a complex $\Sigma$ is a \emph{leaf} if it is the only facet, or if there is a facet $F'\ne F$ such that
	\[
	F\cap Q\subseteq F\cap F'
	\qquad\text{for every facet }Q\ne F.
	\]
	The facet $F'$ is called a branch of $F$.  A \emph{quasi-forest} is a complex whose facets admit an order in which each facet is a leaf of the subcomplex generated by it and the preceding facets.  Equivalently, quasi-forests are clique complexes of chordal graphs \cite{HerzogHibiMuraiTrungZheng2008}.
	
	For $b=\indim\Delta$, let $\mathcal O_b(\Delta)$ be the graph whose vertices are the facets of $\Delta$, with $F_1$ adjacent to $F_2$ when $|F_1\cap F_2|\geq b$.  By definition, $\Delta$ is weakly connected if and only if $\mathcal O_b(\Delta)$ is connected.
	
	\begin{theorem}\label{thm:5}
		Let $\Delta$ be a quasi-forest and put $b=\indim\Delta$.  The following are equivalent:
		\begin{enumerate}[label=\textup{(\roman*)}]
			\item $\Delta$ is weakly connected;
			\item $\Delta$ is strongly vertex dismissible;
			\item $\Delta$ is vertex dismissible;
			\item $\Delta$ is scalable;
			\item $\Delta$ is initially Cohen--Macaulay over some field;
			\item $\Delta$ is initially Cohen--Macaulay over every field.
		\end{enumerate}
	\end{theorem}
	
	\begin{proof}
		The implications already proved give
		\[
		\textup{(ii)}\Longrightarrow\textup{(iii)}
		\Longrightarrow\textup{(iv)}
		\Longrightarrow\textup{(vi)}
		\Longrightarrow\textup{(v)}.
		\]
		A pure Cohen--Macaulay complex is strongly connected in codimension one, so (v) implies (i).  It remains to prove (i) implies (ii).
		
		We induct on the number of vertices.  A quasi-forest with one facet is a simplex.  Suppose that $\Delta$ has at least two facets, put $b=\indim\Delta$, and write $\Delta=\Cl(H)$ for a noncomplete chordal graph $H$.  By Dirac's theorem, $H$ has two nonadjacent simplicial vertices \cite{Dirac1961}.  Each simplicial vertex belongs to a unique maximal clique.  If $\Delta$ has a unique $b$-dimensional facet $F_b$, the two simplicial vertices cannot both belong to $F_b$, because $F_b$ is a clique. Choose a simplicial vertex $x$ outside $F_b$.  If there is more than one $b$-dimensional facet, choose either of the two simplicial vertices.  Let $F$ be the unique facet containing $x$.  In both cases at least one $b$-dimensional facet does not contain $x$ and remains a facet after deletion. Since $x$ is simplicial,
		\[
		\link_\Delta(x)=2^{F\setminus\{x\}}
		\]
		is a simplex.  Every facet of the deletion other than a possible new facet $F\setminus\{x\}$ has dimension at least $b$.  If $|F|\geq b+2$, the possible new facet also has dimension at least $b$.  If $|F|=b+1$, the connectedness of $\mathcal O_b(\Delta)$ gives a facet $F'\ne F$ with $|F\cap F'|\geq b$.  The vertex $x$ occurs only in $F$, so $F\cap F'\subseteq F\setminus\{x\}$.  Both sets have $b$ elements, and hence
		\[
		F\setminus\{x\}=F\cap F'\subseteq F'.
		\]
		Thus $F\setminus\{x\}$ is not a facet of the deletion.  It follows that $x$ is dismissing and that the deletion has initial dimension $b$.
		
		The deletion is a quasi-forest because $H-x$ is chordal. We next verify that the deletion is weakly connected. If $F\setminus\{x\}$ remains a facet, then, since $x$ belongs to no other facet, this new facet has the same neighbors in $\mathcal O_b$ that $F$ had.  If it is contained in a facet $F'$, then every neighbor $F''$ of $F$ is adjacent to $F'$, because
		\[
		F\cap F''\subseteq F\setminus\{x\}\subseteq F'
		\qquad\text{and}\qquad |F\cap F''|\geq b.
		\]
		Therefore deleting or replacing the vertex $F$ does not disconnect the overlap graph.  Thus the deletion again satisfies condition (i) and is strongly vertex dismissible by induction.  Together with the simplex link, this proves that $x$ begins a strong dismissal of $\Delta$.
	\end{proof}
	
	\begin{corollary}\label{cor:4}
		Let $G$ be a co-chordal graph and put $\Delta=\Ind(G)$.  Then the six conditions in Theorem~\ref{thm:5} are equivalent for $\Delta$.
	\end{corollary}
	
	\begin{proof}
		The complement $\overline G$ is chordal and $\Ind(G)=\Cl(\overline G)$.  Hence $\Ind(G)$ is a quasi-forest.
	\end{proof}
	
	\begin{corollary}
		\label{cor:5}
		Let $G=K_{p_1,\ldots,p_t}$ with $p_i\geq1$, put $\rho=\min\{p_1,\ldots,p_t\}$, and set $\Delta=\Ind(G)$.  The six conditions in Theorem~\ref{thm:5} hold if and only if either $t=1$ or $\rho=1$.  In particular, $\Ind(K_{p,q})$ has these properties exactly when $\min\{p,q\}=1$.
	\end{corollary}
	
	\begin{proof}
		The facets of $\Delta$ are the partite sets $V_1,\ldots,V_t$.  If $t=1$, then $\Delta$ is a simplex.  Suppose that $t\geq2$.  Its pure initial skeleton is the disjoint union of the complexes generated by the $\rho$-subsets of the partite sets.  It is weakly connected exactly when $\rho=1$.  Apply Theorem~\ref{thm:5}.
	\end{proof}
	
	\subsection{Independence complexes of pseudoforests}
	\label{subsec:3}
	
	For a finite graph $G$, let
	\[
	i(G)=\min\{|M|:M\text{ is a maximal independent set of }G\}
	\]
	be its independent domination number; see \cite{Haynes1998}. For independence complexes, dismissing vertices are closely related to independent domination vertex critical graphs. Following the vertex deletion terminology in \cite{EdwardsMacGillivray2012,AnanchuenEtAl2018}, a graph $G$ is \emph{independent domination vertex critical}, or \emph{$i$-vertex-critical}, if
	\[
	i(G-v)<i(G)
	\qquad\text{for every }v\in V(G).
	\]
	Thus deleting every vertex of an $i$-vertex-critical graph lowers the independent domination number. By Proposition~\ref{prop:6}, a vertex of $\Ind(G)$ is dismissing precisely when its deletion does not lower $i(G)$. Strong vertex dismissibility therefore requires, at each recursive step, a vertex $v$ with $i(G-v)\geq i(G)$, and such a vertex must exist recursively in both $G-v$ and $G-N_G[v]$. Consequently, an $i$-vertex-critical graph with at least one edge obstructs the recursion.
	
	\begin{proposition}\label{prop:6}
		Let $G$ be a finite graph and let $v\in V(G)$.  Then $\indim\Ind(G)=i(G)-1$, $\del_{\Ind(G)}(v)=\Ind(G-v)$, $\link_{\Ind(G)}(v)=\Ind(G-N_G[v])$, and
		\[
		v\text{ is dismissing in }\Ind(G)
		\quad\Longleftrightarrow\quad
		i(G-v)\geq i(G).
		\]
		Therefore, $\Ind(G)$ is strongly vertex dismissible if and only if $G$ is edgeless, or there is a vertex $v$ satisfying the displayed inequality such that both $\Ind(G-v)$ and $\Ind(G-N_G[v])$ are strongly vertex dismissible.
	\end{proposition}
	
	\begin{proof}
		The facets of $\Ind(G)$ are the maximal independent sets of $G$, so the first equality follows from the definition of $i(G)$.  The deletion and link identities follow directly from the definition of the independence complex.  Therefore
		\[
		\begin{aligned}
			v\text{ dismissing }\Longleftrightarrow
			\indim\Ind(G-v)\geq\indim\Ind(G)\Longleftrightarrow i(G-v)-1\geq i(G)-1.
		\end{aligned}
		\]
		The last statement is the recursive definition written in graph form.
	\end{proof}
	
	\begin{corollary}\label{cor:6}
		Let $G$ be an $i$-vertex-critical graph. Then $\Ind(G)$ is strongly vertex dismissible if and only if $G$ is edgeless. In particular, every $i$-vertex-critical graph with at least one edge is an obstruction to strong vertex dismissibility.
	\end{corollary}
	
	\begin{proof}
		If $G$ is edgeless, then $\Ind(G)$ is a simplex.  Conversely, suppose that $G$ has an edge.  Then $\Ind(G)$ is not a simplex, while $i$-vertex-criticality and Proposition~\ref{prop:6} show that no vertex of $\Ind(G)$ is dismissing.  Hence the strong dismissal recursion cannot begin.
	\end{proof}
	
	\begin{corollary}\label{cor:7}
		If $G=G_1\mathbin{\dot\cup}\cdots\mathbin{\dot\cup}G_s$, then
		\[
		i(G)=\sum_{j=1}^s i(G_j),
		\]
		and $\Ind(G)$ is strongly vertex dismissible if and only if every $\Ind(G_j)$ is strongly vertex dismissible.
	\end{corollary}
	
	\begin{proof}
		A maximal independent set of a disjoint union is the union of maximal independent sets in its components, which proves the formula.  Also,
		\[
		\Ind(G)=\Ind(G_1)*\cdots*\Ind(G_s),
		\]
		so the recursive statement follows from Theorem~\ref{thm:1}.
	\end{proof}
	
	\begin{lemma}\label{lem:5}
		If $\mathcal T$ is a forest, then $\Ind(\mathcal T)$ is vertex decomposable and hence strongly vertex dismissible.
	\end{lemma}
	
	\begin{proof}
		We induct on $|V(\mathcal T)|$.  If $\mathcal T$ has no edges, its independence complex is a simplex. Otherwise choose a leaf $u$ with neighbor $v$.  The vertex $v$ is shedding in $\Ind(\mathcal T)$: every independent set of $\mathcal T-N_{\mathcal T}[v]$ can be enlarged by $u$ inside $\mathcal T-v$.  Also,
		\[
		\del_{\Ind(\mathcal T)}(v)=\Ind(\mathcal T-v),
		\qquad
		\link_{\Ind(\mathcal T)}(v)=\Ind(\mathcal T-N_{\mathcal T}[v]).
		\]
		Both graphs on the right are forests, so induction proves vertex decomposability.  Apply Proposition~\ref{prop:1}.
	\end{proof}
	
	For paths and cycles, with $P_0$ empty,
	\[
	i(P_m)=\left\lceil\frac m3\right\rceil,
	\qquad
	i(C_n)=\left\lceil\frac n3\right\rceil.
	\]
	A selected vertex dominates at most three vertices, while suitable vertices at pairwise distance three attain the bound.
	
	\begin{lemma}\label{lem:6}
		For every integer $q\geq1$ and every field $\K$, the pure complex $\Ind(C_{3q+1})^{[q]}$ is not Cohen--Macaulay over $\K$.
	\end{lemma}
	
	\begin{proof}
		Since $i(C_{3q+1})=q+1$, Lemma~\ref{lem:1} gives
		\[
		\Ind(C_{3q+1})^{[q]}=\Ind(C_{3q+1})^q.
		\]
		Kozlov's computation of cycle independence complexes gives
		\[
		\Ind(C_{3q+1})\simeq S^{q-1}
		\]
		\cite[Proposition~5.2]{Kozlov1999}.  The inclusion of the $q$-skeleton into a simplicial complex induces an isomorphism on reduced homology in every degree strictly below $q$.  Hence
		\[
		\widetilde H_{q-1}\bigl(\Ind(C_{3q+1})^{[q]};\K\bigr)
		\cong
		\widetilde H_{q-1}\bigl(\Ind(C_{3q+1});\K\bigr)
		\cong \K.
		\]
		The complex has dimension $q$. Since $\K$ is arbitrary, Reisner's criterion implies that it is not Cohen--Macaulay over $\K$ for every field $\K$.
	\end{proof}
	
	For a rooted graph $(H,r)$, define costs, with value $+\infty$ for an impossible state, by
	\begin{align*}
		\alpha(H,r)&=\min\{|D|-1:r\in D,\ D\text{ independently dominates }H\},\\
		\beta(H,r)&=\min\{|D|:r\notin D,\ D\text{ independently dominates }H\},\\
		\gamma(H,r)&=\min\{|D|:r\notin D,\ D\text{ is independent, dominates }H-r,\text{ and does not dominate }r\}.
	\end{align*}
	
	\begin{lemma}\label{lem:7}
		For every rooted graph $(H,r)$, we have $i(H)=\min\{1+\alpha(H,r),\beta(H,r)\}$, $i(H-r)=\min\{\beta(H,r),\gamma(H,r)\}$, and $\alpha(H,r)\leq \gamma(H,r)$.  Therefore,
		\[
		r\text{ is not dismissing in }\Ind(H)
		\quad\Longleftrightarrow\quad
		\gamma(H,r)=\alpha(H,r)<\beta(H,r).
		\]
		If two rooted graphs $H_1$ and $H_2$ meet only at their common root $r$, then $r$ is not dismissing in $\Ind(H_1\cup H_2)$ if and only if it is not dismissing in both $\Ind(H_1)$ and $\Ind(H_2)$.
	\end{lemma}
	
	\begin{proof}
		The first two formulas separate the allowed sets according to the state of the root. A set counted by $\gamma(H,r)$ contains no neighbor of $r$, so adjoining $r$ produces a set counted by $\alpha(H,r)$; hence $\alpha\leq\gamma$.  The condition that $r$ is not dismissing is
		\[
		\min\{\beta,\gamma\}<\min\{1+\alpha,\beta\},
		\]
		and, because these quantities are integers and $\alpha\leq\gamma$, this is equivalent to $\gamma=\alpha<\beta$.
		
		For the one-sum of two rooted graphs, write their costs as $\alpha_j,\beta_j,\gamma_j$ for $j=1,2$.  Then
		\[
		\alpha=\alpha_1+\alpha_2,
		\qquad
		\gamma=\gamma_1+\gamma_2,
		\]
		and
		\[
		\beta=\min\{\beta_1+\beta_2,\beta_1+\gamma_2,\gamma_1+\beta_2\}.
		\]
		Thus $\gamma=\alpha<\beta$ holds exactly when $\gamma_j=\alpha_j<\beta_j$ for both summands.
	\end{proof}
	
	\begin{proposition}\label{prop:7}
		Let $G$ be a connected unicyclic graph with unique cycle $C=c_0c_1\cdots c_{m-1}c_0$. After deleting the cycle edges, let $(T_j,c_j)$ be the rooted tree component containing $c_j$, and put $\Delta=\Ind(G)$.  The following are equivalent:
		\begin{enumerate}[label=\textup{(\roman*)}]
			\item either $m\not\equiv1\pmod3$, or $m=3q+1$ and some $c_j$ is dismissing in $\Ind(T_j)$;
			\item $\Delta$ is strongly vertex dismissible;
			\item $\Delta$ is vertex dismissible;
			\item $\Delta$ is scalable;
			\item $\Delta$ is initially Cohen--Macaulay over some field;
			\item $\Delta$ is initially Cohen--Macaulay over every field.
		\end{enumerate}
	\end{proposition}
	
	\begin{proof}
		We first prove (i) implies (ii).  Suppose that $m\not\equiv1\pmod3$ and induct on $|V(G)|$.  If $G=C_m$, then for every $v\in V(C_m)$,
		\[
		\del_{\Ind(C_m)}(v)=\Ind(P_{m-1}),
		\qquad
		\link_{\Ind(C_m)}(v)=\Ind(P_{m-3}),
		\]
		and
		\[
		i(C_m-v)=\left\lceil\frac{m-1}{3}\right\rceil
		\geq \left\lceil\frac m3\right\rceil=i(C_m)
		\]
		exactly when $m\equiv0,2\pmod3$.  Thus $v$ is dismissing, and both branches are strongly vertex dismissible by Lemma~\ref{lem:5}.
		
		If $G$ is not a cycle, choose a leaf $u$ with neighbor $v$.  The vertex $v$ is shedding in $\Ind(G)$ and therefore dismissing by Corollary~\ref{cor:1}.  Every component of $G-v$ and $G-N_G[v]$ is a forest or a smaller connected unicyclic graph with the same cycle.  The induction hypothesis, Lemma~\ref{lem:5}, and Corollary~\ref{cor:7} show that both resulting independence complexes are strongly vertex dismissible.
		
		Now let $m=3q+1$ and suppose that some $c_j$ is dismissing in $\Ind(T_j)$.  View $G$ as the rooted one-sum of $T_j$ and the rest of $G$, both rooted at $c_j$.  By Lemma~\ref{lem:7}, the vertex $c_j$ is dismissing in $\Ind(G)$.  Both $G-c_j$ and $G-N_G[c_j]$ are forests.  Apply Lemma~\ref{lem:5} and Corollary~\ref{cor:7}.
		
		The implications already proved give
		\[
		\textup{(ii)}\Longrightarrow\textup{(iii)}
		\Longrightarrow\textup{(iv)}
		\Longrightarrow\textup{(vi)}
		\Longrightarrow\textup{(v)}.
		\]
		Assume that (i) fails.  Then $m=3q+1$ and every $c_j$ is not dismissing in $\Ind(T_j)$.  Write
		\[
		\alpha_j=\alpha(T_j,c_j),\qquad
		\beta_j=\beta(T_j,c_j),\qquad
		\gamma_j=\gamma(T_j,c_j).
		\]
		By Lemma~\ref{lem:7}, $\gamma_j=\alpha_j<\beta_j$.  Put $k_j=\alpha_j$ and $\kappa=\sum_jk_j$.  We claim that
		\begin{equation}\label{eq:1}
			i(G)=\kappa+q+1.
		\end{equation}
		A selected root contributes at least $k_j+1$, an excluded root dominated inside its tree contributes at least $\beta_j\geq k_j+1$, and a root undominated inside its tree contributes at least $\gamma_j=k_j$ and must have a selected cycle neighbor.  If $\mathcal S$ is the set of selected roots and $\mathcal B$ is the set of roots dominated internally, then $\mathcal S$ is independent and every root outside $\mathcal S\cup\mathcal B$ has a neighbor in $\mathcal S$.  Therefore
		\[
		3q+1\leq3|\mathcal S|+|\mathcal B|,
		\]
		so $|\mathcal S|+|\mathcal B|\geq q+1$.  This proves the lower bound in~\eqref{eq:1}. For the reverse inequality, combine an independent dominating set of $C_{3q+1}$ of size $q+1$ with minimum sets of state $\alpha$ at the selected roots and minimum sets of state $\gamma$ at the remaining roots.
		
		Choose minimum sets $D_j$ of state $\gamma$ and put $D=\bigcup_jD_j$.  Then $D$ is independent, $|D|=\kappa$, it dominates every vertex outside the cycle, and it dominates no cycle vertex.  Hence
		\[
		G-N_G[D]=C_{3q+1}.
		\]
		Let $\Gamma=\Delta^{[i(G)-1]}$.  Since $|D|=\kappa$ and $i(G)=\kappa+q+1$, a facet of $\link_\Gamma(D)$ is obtained exactly by adjoining to $D$ an independent $(q+1)$-set of $G-N_G[D]=C_{3q+1}$.  Conversely, every independent $(q+1)$-set of this remaining cycle gives an independent set of $G$ of size $i(G)$ and hence a facet of $\Gamma$ containing $D$.  Therefore
		\[
		\link_\Gamma(D)=\Ind(C_{3q+1})^{[q]}.
		\]
		By Lemma~\ref{lem:6}, this link is not Cohen--Macaulay over $\K$ for every field $\K$.  Hence $\Gamma$ is not Cohen--Macaulay, contradicting (v).  Thus (v) implies (i).
	\end{proof}
	
	A graph is a \emph{pseudoforest} if every connected component contains at most one cycle.
	
	\begin{theorem}\label{thm:6}
		Let $G$ be a finite pseudoforest.  For every unicyclic component $U$, write its unique cycle as $c_0c_1\cdots c_{m_U-1}c_0$, and let $(T_{U,j},c_j)$ be the rooted tree obtained after deleting the cycle edges.  The following are equivalent:
		\begin{enumerate}[label=\textup{(\roman*)}]
			\item for every unicyclic component $U$, either $m_U\not\equiv1\pmod3$, or some $c_j$ is dismissing in $\Ind(T_{U,j})$;
			\item $\Ind(G)$ is strongly vertex dismissible;
			\item $\Ind(G)$ is vertex dismissible;
			\item $\Ind(G)$ is scalable;
			\item $\Ind(G)$ is initially Cohen--Macaulay over some field;
			\item $\Ind(G)$ is initially Cohen--Macaulay over every field.
		\end{enumerate}
	\end{theorem}
	
	\begin{proof}
		Tree components have vertex decomposable independence complexes by Lemma~\ref{lem:5}.  The connected unicyclic components are classified by Proposition~\ref{prop:7}.  If (i) holds, every component is strongly vertex dismissible and their join is strongly vertex dismissible by Corollary~\ref{cor:7}.  The implications already proved give
		\[
		\textup{(ii)}\Longrightarrow\textup{(iii)}
		\Longrightarrow\textup{(iv)}
		\Longrightarrow\textup{(vi)}
		\Longrightarrow\textup{(v)}.
		\]
		Conversely, Proposition~\ref{prop:5} shows that these three properties are inherited by each factor of a join.  For each unicyclic component, each of them implies condition (i) by Proposition~\ref{prop:7}.
	\end{proof}
	
	\begin{corollary}\label{cor:8}
		Let $m\geq3$ and put $\Delta_m=\Ind(C_m)$.  The following are equivalent:
		\begin{enumerate}[label=\textup{(\roman*)}]
			\item $m\not\equiv1\pmod3$;
			\item $\Delta_m$ is strongly vertex dismissible;
			\item $\Delta_m$ is vertex dismissible;
			\item $\Delta_m$ is scalable;
			\item $\Delta_m$ is initially Cohen--Macaulay over some field;
			\item $\Delta_m$ is initially Cohen--Macaulay over every field.
		\end{enumerate}
	\end{corollary}
	
	\begin{proof}
		In Proposition~\ref{prop:7}, every rooted tree attached to the cycle is the graph with one vertex.  Its root is not dismissing because deleting it reduces the independent domination number from one to zero.  The condition in Proposition~\ref{prop:7}(i) therefore reduces to $m\not\equiv1\pmod3$.
	\end{proof}
	
	\begin{remark}
		Wu et al.\ \cite[Corollary~1]{WuEtAl2019} proved that the cycle $C_m$ is independent domination stable exactly when $m\not\equiv1\pmod3$. Thus the residue pattern in Corollary~\ref{cor:8} was already known for independent domination stability. In the present setting, the same residue condition also characterizes strong vertex dismissibility and the three corresponding properties of the pure initial skeleton of the independence complex.
	\end{remark}
	
	\begin{corollary}\label{cor:9}
		Let $G$ be a graph of maximum degree at most two.  The six properties in Theorem~\ref{thm:6} hold if and only if no connected component is a cycle $C_{3q+1}$.  Therefore, $C_4,C_7,C_{10},\ldots$ are exactly the minimal obstructions with respect to induced subgraphs in this graph class.
	\end{corollary}
	
	\begin{proof}
		Every component is a path or a cycle.  Apply Lemma~\ref{lem:5} and Corollary~\ref{cor:8} on each component.  Every proper induced subgraph of a cycle is a forest.
	\end{proof}
	
	\section{Alexander dual ideals and degree conditions}\label{sec:6}
	
	Let $X$ be a fixed $n$-element ground set and let $R=\K[x_v:v\in X]$. Unless a smaller ground set is displayed explicitly, Alexander duals in this section are taken with respect to $X$. For a monomial ideal $I\subseteq R$ and an integer $\ell\geq0$, define
	\[
	I_{\ell}
	=\bigl\langle u\in I:u\text{ is squarefree and }\deg u=\ell\bigr\rangle.
	\]
	Thus $I_{\ell}$ is the ideal generated by the squarefree monomials of degree $\ell$.  It is not the degree-$\ell$ vector space component of $I$.  This notation also distinguishes the ideal subscript from the brackets used for pure skeleta.
	
	\begin{lemma}[cf. {\cite[Lemma 4.1]{Namiq2026ICM}}]\label{lem:8}
		Let $\Delta$ be a complex on $X$, let $b=\indim\Delta$, and put $J=I_{\Delta^{\vee}}$ and $d=n-b-1$. For every $-1\leq k\leq b$, with $\ell=n-k-1$, one has
		\[
		J_{\ell}=I_{(\Delta^{[k]})^{\vee}}
		=I_{(\Delta^k)^{\vee}}.
		\]
		In particular, $J_d=I_{(\Delta^{[b]})^{\vee}}$.
	\end{lemma}
	
	\begin{proof}
		A squarefree monomial of degree $\ell=n-k-1$ has the form $x^{X\setminus A}$ for a $(k+1)$-element subset $A\subseteq X$.  Since
		\[
		J=(x^{X\setminus F}:F\in\F(\Delta)),
		\]
		we have
		\[
		x^{X\setminus A}\in J
		\quad\Longleftrightarrow\quad
		A\subseteq F\text{ for some }F\in\F(\Delta)
		\quad\Longleftrightarrow\quad
		A\in\Delta.
		\]
		Thus the squarefree degree-$\ell$ monomials in $J$ are exactly the complements of the $k$-faces of $\Delta$.  Those $k$-faces are the facets of $\Delta^{[k]}$, so their complements are precisely the minimal generators of $I_{(\Delta^{[k]})^{\vee}}$.  The second equality follows from Lemma~\ref{lem:1}.
	\end{proof}
	
	A monomial ideal is \emph{vertex splittable} if it is $0$, $R$, or principal, or it has a recursive decomposition
	\[
	I=xI_1+I_2,
	\qquad I_2\subseteq I_1,
	\qquad
	\mathcal G(I)=x\mathcal G(I_1)\,\dot\cup\,\mathcal G(I_2),
	\]
	where $I_1$ and $I_2$ are vertex splittable ideals in the polynomial ring without $x$ \cite{MoradiKhoshAhang2016}.  Linear quotients have their usual meaning.
	
	For a nonzero monomial ideal $I$, write
	\[
	\deg I=\max\{\deg u:u\in\mathcal G(I)\}.
	\]
	We use the following terminology. The terms \emph{vertex divisible} and \emph{degree quotients} refer to the corresponding conditions on the squarefree ideal in the maximum generator degree. The term \emph{degree resolution} was introduced in \cite[Definition~3.1]{Namiq2026Skeletons}, and we use it in the same sense.
	\begin{itemize}
		\item A squarefree monomial ideal $I$ with $d=\deg I$ is \emph{vertex divisible} if $I_d$ is vertex splittable.
		\item A squarefree monomial ideal $I$ with $d=\deg I$ \emph{has degree quotients} if $I_d$ has linear quotients.  Thus this condition is imposed on the squarefree degree-$d$ ideal rather than on the ideal itself.
		\item Following \cite[Definition~3.1]{Namiq2026Skeletons}, a monomial ideal $I$ \emph{has a degree resolution} if $\reg I=\deg I$. Equivalently, its graded Betti numbers satisfy $\beta_{i,j}(I)=0$ for $j>\deg I+i$.
	\end{itemize}
	
	We also use the following recursive notion. A squarefree monomial ideal $I\subseteq\K[x_v:v\in Y]$ is \emph{strongly vertex divisible} if it is $0$, the unit ideal, or principal, or if there is a variable $x$ such that
	\[
	\deg(I:x)\leq\deg I-1
	\]
	and both $(I:x)$ and $I\cap\K[x_v:v\in Y\setminus\{x\}]$ are strongly vertex divisible in the smaller polynomial ring.  The displayed condition calls $x$ a \emph{dividing variable}.
	
	If $\Delta$ is a complex on the ground set $Y$ and $I=I_{\Delta^{\vee_Y}}$, then complementing facets gives
	\[
	I=x(I:x)+\bigl(I\cap\K[x_v:v\in Y\setminus\{x\}]\bigr),
	\]
	where the two smaller ideals are the Alexander dual ideals, taken over $Y\setminus\{x\}$, of $\del_\Delta(x)$ and $\link_\Delta(x)$, respectively.  Moreover, $x$ is a dismissing vertex of $\Delta$ exactly when $\deg(I:x)\leq\deg I-1$.  Induction therefore gives
	\[
	\Delta\text{ is strongly vertex dismissible}
	\quad\Longleftrightarrow\quad
	I_{\Delta^{\vee_Y}}\text{ is strongly vertex divisible}.
	\]
	This recursive statement concerns the Alexander dual ideal itself. By contrast, vertex divisibility and the criterion for scalability below depend only on the squarefree ideal in the maximum generator degree.
	
	\begin{lemma}\label{lem:9}
		Let $S=\K[x_v:v\in Y]$, let $S'=S[y_1,\ldots,y_s]$, put $\mu=y_1\cdots y_s$, and let $I\subseteq S$ be a monomial ideal.
		\begin{enumerate}[label=\textup{(\roman*)}]
			\item $I$ is vertex splittable if and only if $\mu I\subseteq S'$ is vertex splittable.
			\item $I$ has linear quotients if and only if $\mu I$ has linear quotients.
			\item If $I$ is generated in degree $\ell$, then $I$ has an $\ell$-linear resolution if and only if $\mu I$ has an $(\ell+s)$-linear resolution.
		\end{enumerate}
	\end{lemma}
	
	\begin{proof}
		It is enough to prove the statement after adjoining one variable $y$.  If $I$ is vertex splittable, then $yI=yI+0$ is a vertex splitting with splitting variable $y$.  Conversely, factor $y$ from a vertex splitting tree for $yI$.  If the root splits on $y$, the second branch is $0$ and the first branch is $I$.  If the root splits on an old variable, every minimal generator in both branches is divisible by $y$; induction on the height of the splitting tree factors $y$ from the two branches and gives a vertex splitting of $I$.  This proves (i).
		
		For a minimal generating order $u_1,\ldots,u_t$ of $I$, the order $yu_1,\ldots,yu_t$ satisfies
		\[
		(yu_1,\ldots,yu_{j-1}):yu_j
		=(u_1,\ldots,u_{j-1}):u_j,
		\]
		so (ii) follows in both directions.  Finally, $yI\cong (I\otimes_S S')(-1)$ as graded $S'$-modules, so its graded Betti numbers are those of $I$ with every internal degree shifted by one.  This proves (iii).
	\end{proof}
	
	Part~(iii) of the next theorem uses terminology from earlier work: initially Cohen--Macaulayness is from \cite{Namiq2026ICM}, and degree resolutions are from \cite{Namiq2026Skeletons}. We include these known equivalences alongside the vertex divisibility and scalability correspondences in parts~(i)--(ii).
	
	\begin{theorem}\label{thm:7}
		With the notation of Lemma~\ref{lem:8}, the following statements hold.
		\begin{enumerate}[label=\textup{(\roman*)}]
			\item $\Delta$ is vertex dismissible if and only if $J$ is vertex divisible, equivalently, $J_d$ is vertex splittable.
			\item $\Delta$ is scalable if and only if $J$ has degree quotients, equivalently, $J_d$ has linear quotients.
			\item $\Delta$ is initially Cohen--Macaulay over $\K$ if and only if $J$ has a degree resolution, equivalently, $J_d$ has a $d$-linear resolution over $\K$.
		\end{enumerate}
	\end{theorem}
	
	\begin{proof}
		Put $\Gamma=\Delta^{[b]}$, let $Y=V(\Gamma)$, and set $Z=X\setminus Y$. By Lemma~\ref{lem:8},
		\[
		J_d=I_{\Gamma^{\vee}}
		=x^Z I_{\Gamma^{\vee_Y}}.
		\]
		Degrees over $X$ are obtained from degrees over $Y$ by adding $|Z|$.  Thus Lemma~\ref{lem:9} reduces all three statements to the polynomial ring on $Y=V(\Gamma)$.  Moradi and Khosh-Ahang proved that a simplicial complex is vertex decomposable if and only if the Stanley--Reisner ideal of its Alexander dual is vertex splittable \cite[Theorem~2.3]{MoradiKhoshAhang2016}.  This gives (i).
		
		For (ii), by definition $J$ has degree quotients exactly when $J_d$ has linear quotients.  Let $F_1,\dots,F_t$ be the facets of the pure complex $\Gamma$ and let $u_i=x^{Y\setminus F_i}$.  For each $j>1$,
		\[
		(u_1,\dots,u_{j-1}):u_j
		=\bigl(x^{F_j\setminus F_i}:i<j\bigr).
		\]
		This colon ideal is generated by variables if and only if, for every $i<j$, there is $k<j$ such that $F_j\setminus F_k$ is a singleton contained in $F_j\setminus F_i$.  This is exactly the pure shelling criterion for the order $F_1,\dots,F_t$.
		
		Part~(iii) restates earlier results in the notation used here. The equivalence between initially Cohen--Macaulayness and a degree resolution of the Alexander dual ideal itself is \cite[Proposition~4.4]{Namiq2026ICM}. Equivalently, \cite[Proposition~4.6]{Namiq2026ICM} identifies initially Cohen--Macaulayness with Cohen--Macaulayness of $\Gamma$, and the Eagon--Reiner theorem \cite{EagonReiner1998} identifies the latter with a $d$-linear resolution of $J_d$. Finally, \cite[Proposition~3.3]{Namiq2026Skeletons} gives the equivalence between a degree resolution of $J$ and a linear resolution of the squarefree degree-$d$ ideal $J_d$.
	\end{proof}
	
	\begin{corollary}\label{cor:11}
		If $\Delta$ is strongly vertex dismissible, then for every $-1\leq k\leq b$ the ideal $J_{n-k-1}$ is vertex splittable, has linear quotients, and has an $(n-k-1)$-linear resolution over every field.
	\end{corollary}
	
	\begin{proof}
		Apply Theorem~\ref{thm:3}, Lemma~\ref{lem:8}, and Theorem~\ref{thm:7} to $\Delta^{[k]}$.
	\end{proof}
	
	\begin{remark}\label{rem:5}
		Theorem~\ref{thm:7}(i)--(ii) concerns the squarefree ideal $J_d$ in degree $d$. In particular, vertex divisibility of $J$ does not mean that $J$ itself is vertex splittable, and degree quotients of $J$ do not mean that $J$ itself has linear quotients. These conditions in degree $d$ also do not determine the strong recursion.  Example~\ref{ex:4} has a vertex splittable $J_d$ but no strong dismissal.  Strong vertex divisibility is the separate recursive property of the Alexander dual ideal itself defined above. By contrast, the degree resolution condition in Theorem~\ref{thm:7}(iii) is a property of $J$ itself.
	\end{remark}
	
	\section{Depth and projective dimension}\label{sec:7}
	
	Let $\K[\Delta]=R/I_\Delta$.  We use the following classical criterion relating depth to skeleta. It also makes the dependence on the field and the grading shifts explicit.
	
	\begin{lemma}\cite{HerzogSoleymanJahanZheng2010}\label{lem:depth}
		Let $b=\indim\Delta$.  For every $-1\leq k\leq b$, $\Delta^k$ is Cohen--Macaulay over $\K$ if and only if $\depth\K[\Delta]\geq k+1$. Therefore,
		\[
		\depth\K[\Delta]
		=1+\max\{-1\leq k\leq b:\Delta^k
		\text{ is Cohen--Macaulay over }\K\}.
		\]
	\end{lemma}
	
	The next corollary collects the earlier characterizations of initially Cohen--Macaulayness and degree resolutions in the notation of this paper, together with the standard Auslander--Buchsbaum formulation. The equivalence between initially Cohen--Macaulayness and degree resolutions is from \cite{Namiq2026ICM}, and the notion of a degree resolution is from \cite{Namiq2026Skeletons}.
	
	\begin{corollary}\label{cor:12}
		Let $X$ have $n$ elements, let $b=\indim\Delta$, put $d=n-b-1$, and set $J=I_{\Delta^{\vee}}$.  The following are equivalent:
		\begin{enumerate}[label=\textup{(\roman*)}]
			\item $\Delta$ is initially Cohen--Macaulay over $\K$;
			\item $\depth\K[\Delta]=b+1$;
			\item $\pdim\K[\Delta]=n-b-1=d$;
			\item $J_d$ has a $d$-linear resolution over $\K$;
			\item $J$ has a degree resolution, equivalently $\reg J=d$.
		\end{enumerate}
	\end{corollary}
	
	\begin{proof}
		The equivalence of (i) and the inequality $\depth\K[\Delta]\geq b+1$ is Lemma~\ref{lem:depth}.  Every facet of minimum dimension $F$ gives a minimal prime of $\K[\Delta]$ of dimension $|F|=b+1$, so $\depth\K[\Delta]\leq b+1$.  This proves (i)$\Longleftrightarrow$(ii). Auslander--Buchsbaum gives (ii)$\Longleftrightarrow$(iii).  The equivalence with (iv) is Theorem~\ref{thm:7}(iii).  Finally, Terai's Alexander dual equality
		\[
		\reg I_{\Delta^{\vee}}=\pdim\K[\Delta]
		\]
		gives (iii)$\Longleftrightarrow$(v) \cite{Terai1999}.
	\end{proof}
	
	\begin{remark}
		The number $d$ is the maximum degree of a minimal generator of $J$, because those generators are $x^{X\setminus F}$ with $F$ a facet of $\Delta$ and $\min\{|F|:F\in\F(\Delta)\}=b+1$.  Thus condition (v), the degree resolution condition, is a numerical statement about the regularity of the Alexander dual ideal itself.  It does not say that $J$ has a $d$-linear resolution when generators of lower degree occur, and it does not transfer the vertex splitting or linear quotient properties of $J_d$ to $J$.
	\end{remark}
	
	\begin{corollary}\label{cor:13}
		Let $X$ have $n$ elements and let $b=\indim\Delta$.  If $\Delta$ is strongly vertex dismissible, vertex dismissible, or scalable, then for every field $\K$,
		\[
		\depth\K[\Delta]=b+1,
		\qquad
		\pdim\K[\Delta]=n-b-1.
		\]
	\end{corollary}
	
	\begin{proof}
		Each stated hypothesis makes $\Delta^{[b]}$ Cohen--Macaulay over every field. Apply Corollary~\ref{cor:12}.
	\end{proof}
	
	\begin{remark}\label{rem:6}
		The exact formulas in Corollary~\ref{cor:12} are field dependent in general, because Cohen--Macaulayness of a pure complex can depend on the characteristic.  The hypotheses in Corollary~\ref{cor:13} remove that dependence by providing a combinatorial shelling or vertex decomposition. For other Alexander dual ideals, the ring, ground set, and grading shift must be fixed before determining projective dimension or regularity.
	\end{remark}
	
	\section{Limitations and open problems}\label{sec:8}
	
	Strong vertex dismissibility is not preserved under links.
	
	\begin{example}\label{ex:5}
		Let $\Delta=\langle de,ace,bcd\rangle$. The vertex $a$ is dismissing, its link is the simplex $\langle ce\rangle$, and its deletion is $\langle de,ce,bcd\rangle$. In that deletion, $b$ is dismissing with simplex link $\langle cd\rangle$, and the remaining deletion $\langle de,ce,cd\rangle$ is a connected pure graph. Hence $\Delta$ is strongly vertex dismissible. However,
		\[
		\link_\Delta(c)=\langle ae,bd\rangle
		\]
		is a disjoint union of two edges. It is pure and not vertex decomposable, so it is not strongly vertex dismissible by Proposition~\ref{prop:2}.
	\end{example}
	
	The link is strongly vertex dismissible in at least the following three situations. If $\Delta$ is vertex decomposable in the usual sense, then all its links are vertex decomposable and hence strongly vertex dismissible. If a face is obtained by successive links along a fixed strong dismissal sequence, its link is strongly vertex dismissible by construction. Finally, if $\Delta$ is strongly vertex dismissible and $\Delta=2^\sigma*\Gamma$, then $\link_\Delta(\sigma)=\Gamma$, and Theorem~\ref{thm:1} applies to the two factors. These conditions are sufficient, but they do not cover every such face.
	
	\begin{question}\label{q:1}
		For which natural classes does vertex decomposability of the pure initial skeleton imply strong vertex dismissibility of the complex itself? Can such classes be characterized by elimination orders, forbidden induced subcomplexes, or restrictions on the intersections of facets of larger dimension with the pure initial skeleton?
	\end{question}
	
	The pure case and Theorem~\ref{thm:2} give classes for which the implication holds; Example~\ref{ex:4} shows that the converse fails in general. The main obstruction is that the pure initial skeleton does not determine the links arising from facets of larger dimension.
	
	The recursive Alexander dual description of strong vertex divisibility leads to a second problem. Let $S$ be a polynomial ring, let $x$ be a variable not occurring in $S$, and suppose that $x$ is a dividing variable for a strongly vertex divisible ideal $I\subseteq S[x]$. Put
	\[
	J=(I:x)
	\qquad\text{and}\qquad
	K=I\cap S.
	\]
	Then $K\subseteq J$ and
	\[
	I=xJ+K.
	\]
	Unlike a vertex splitting, this decomposition need not give a Betti splitting, since the inclusion $K\hookrightarrow J$ may induce nonzero maps on Tor. Define the \emph{homological redundancy} by
	\[
	\delta_{i,j}
	=
	\dim_{\K}\operatorname{Im}\!\left(
	\operatorname{Tor}_i^S(K,\K)_j
	\longrightarrow
	\operatorname{Tor}_i^S(J,\K)_j
	\right).
	\]
	The short exact sequence
	\[
	0\longrightarrow xK
	\longrightarrow xJ\oplus K
	\longrightarrow I
	\longrightarrow0
	\]
	then gives
	\[
	\beta_{i,j}(I)
	=
	\beta_{i,j-1}(J)
	+\beta_{i,j}(K)
	+\beta_{i-1,j-1}(K)
	-\delta_{i,j-1}
	-\delta_{i-1,j-1},
	\]
	where the Betti numbers of $J$ and $K$ are computed over $S$, with $\beta_{-1,*}(K)=\delta_{-1,*}=0$.
	
	\begin{question}\label{q:homological-redundancy}
		Can the homological redundancy $\delta_{i,j}$ be described combinatorially from the recursive data of a strongly vertex divisible ideal? In particular, is there a combinatorial formula for the correction terms $\delta_{i,j-1}$ and $\delta_{i-1,j-1}$ in terms of the corresponding deletion--link data or a strong dismissal sequence?
	\end{question}
	
	For a finite complex $\Sigma$, taking the maxima over $-1\leq k\leq\indim\Sigma$, set
	\[
	s(\Sigma)=\max\{k:\Sigma^k\text{ is shellable}\},\qquad
	c_{\K}(\Sigma)=\max\{k:\Sigma^k\text{ is Cohen--Macaulay over }\K\}.
	\]
	The two thresholds can differ, and the second can depend on the characteristic. For example, take $\Sigma=\Gamma_1$ from Example~\ref{ex:3}. Its one-skeleton is connected and therefore shellable and Cohen--Macaulay over every field, while $\Gamma_1$ itself is not shellable and is Cohen--Macaulay exactly when $\operatorname{char}\K\ne2$. Therefore
	\[
	s(\Gamma_1)=1,\qquad
	c_{\K}(\Gamma_1)=
	\begin{cases}
		2,&\operatorname{char}\K\ne2,\\
		1,&\operatorname{char}\K=2.
	\end{cases}
	\]
	
	\begin{question}\label{q:thresholds}
		What are the smallest complexes for which $s(\Sigma)<c_{\K}(\Sigma)$, and which structural restrictions imply equality? How complicated can the characteristic dependence of $c_{\K}(\Sigma)$ be?
	\end{question}

\end{document}